\documentclass[11pt, reqno]{amsart}
\pdfoutput=1

\usepackage{amssymb}
\usepackage{mathtools}
\usepackage{mathrsfs}
\usepackage{stmaryrd}
\usepackage{upgreek}
\usepackage{centernot}
\usepackage{cancel}

\usepackage{bm}

\usepackage[T2A,T1]{fontenc}
\usepackage[utf8]{inputenc}
\usepackage[russian, french, english]{babel}
\usepackage{lmodern}
\usepackage[spacing=true,kerning=true,babel=true,nopatch=footnote]{microtype}
\usepackage[shortcuts]{extdash}
\usepackage[inline]{enumitem}
\usepackage{indentfirst}
\usepackage{xspace}
\usepackage{xcolor}
\usepackage{csquotes}

\usepackage[a4paper, left=2cm, right=2cm, top=1in, bottom=1in]{geometry}
\usepackage{graphicx}
\usepackage{wrapfig}
\usepackage{float}
\usepackage{array}
\usepackage{multicol}
\usepackage{mdframed}
\usepackage[skins]{tcolorbox}
\usepackage{framed}
\usepackage[justification=centering, labelfont=bf, font=small]{caption}
\usepackage{algorithm}

\usepackage{tikz}
\usetikzlibrary{shapes}
\usetikzlibrary{arrows.meta}
\usetikzlibrary{decorations.pathmorphing}
\usetikzlibrary{patterns}

\usepackage{titlesec}
\usepackage{titletoc}

\usepackage[
    sortcites,
    backend=biber,
    style=alphabetic,
    sorting=nyt,
    maxnames=100,
    backref=false
]{biblatex}
\usepackage[hidelinks]{hyperref}
\usepackage{cleveref}

\newcommand{\bemph}[1]{{\normalfont#1}}
\newcommand{\ep}[1]{\bemph{(}#1\bemph{)}}

\newtheoremstyle{bfnote}{}{}{\slshape}{}{\bfseries}{\bfseries.}{ }{\thmname{#1}\thmnumber{ #2}\thmnote{ \ep{\normalfont{}#3}}}
\theoremstyle{bfnote}
\newtheorem{theo}{Theorem}[section]
\newtheorem*{theo*}{Theorem}
\newtheorem{prop}[theo]{Proposition}
\newtheorem{lemma}[theo]{Lemma}
\newtheorem{claim}[theo]{Claim}

\newtheorem{corl}[theo]{Corollary}

\newtheorem*{corl*}{Corollary}

\theoremstyle{definition}

\newtheorem*{defn*}{Definition}

\newtheorem*{exmp*}{Example}

\theoremstyle{remark}
\newtheorem*{ques*}{Question}
\newtheorem*{remk*}{Remark}

\crefname{theo}{theorem}{theorems}
\Crefname{theo}{Theorem}{Theorems}
\crefname{lemma}{lemma}{lemmas}
\Crefname{lemma}{Lemma}{Lemmas}
\crefname{claim}{claim}{claims}
\Crefname{claim}{Claim}{Claims}
\crefname{prop}{proposition}{propositions}
\Crefname{prop}{Proposition}{Propositions}
\crefname{corl}{corollary}{corollaries}
\Crefname{corl}{Corollary}{Corollaries}
\crefname{defn}{definition}{definitions}
\Crefname{defn}{Definition}{Definitions}
\crefname{remk}{remark}{remarks}
\Crefname{remk}{Remark}{Remarks}
\crefname{alg}{algorithm}{algorithms}
\Crefname{alg}{Algorithm}{Algorithms}

\newcommand*{\myproofname}{Proof}
\newenvironment{claimproof}[1][\myproofname]{\begin{proof}[#1]}{\end{proof}}

\makeatletter
\newcommand{\neutralize}[1]{\expandafter\let\csname c@#1\endcsname\count@}
\newcommand{\authorfootnote}[1]{%
  \begingroup
  \renewcommand{\@makefnmark}{}%
  \renewcommand{\@makefntext}[1]{\noindent ##1}%
  \footnotetext{#1}%
  \endgroup
}
\makeatother

\newcommand{\0}{\varnothing}
\newcommand{\B}{\mathcal{B}}
\newcommand{\set}[1]{\{#1\}}

\newcommand{\F}{\mathcal{F}}

\newcommand{\D}{\mathcal{D}}

\renewcommand{\P}{\mathbb{P}}
\newcommand{\E}{\mathbb{E}}
\newcommand{\tE}{\widetilde{\mathbb E}}

\renewcommand{\epsilon}{\varepsilon}
\newcommand{\eps}{\varepsilon}
\renewcommand{\phi}{\varphi}
\renewcommand{\leq}{\leqslant}
\renewcommand{\geq}{\geqslant}
\newcommand{\defeq}{\coloneqq}

\newcommand{\red}[1]{{\color{red}#1}}

\newcommand{\todo}[1]{\red{\textbf{\upshape{[#1]}}}}
\newcommand{\emphd}[1]{{\fontseries{b}\selectfont\textsf{#1}}}

\newcommand{\indicator}{\mathbf{1}}

\numberwithin{equation}{section}

\titleformat{\section}[block]{\large\bfseries\sffamily}{\thesection.}{1ex}{}
\titleformat{\subsection}[block]{\bfseries\sffamily}{\thesubsection.}{1ex}{}
\titleformat{\subsubsection}[block]{\itshape}{\bfseries\upshape\sffamily\thesubsubsection.}{1ex}{}

\titlespacing*{\section}{0pt}{*3}{*1}
\titlespacing*{\subsection}{0pt}{*3}{*1}
\titlespacing*{\subsubsection}{0pt}{*2}{*1}

\titlecontents{section}[1.5em]{\smallskip}{\bfseries\thecontentslabel\hspace{1.02em}}{\bfseries}{\,\,\titlerule*[0.77pc]{}\bfseries\contentspage}
\titlecontents{subsection}[4em]{\smallskip}{\thecontentslabel\hspace{1.02em}}{\hspace*{2.32em}}{\,\,\titlerule*[0.77pc]{.}\contentspage}

\renewbibmacro{in:}{}
\renewbibmacro*{volume+number+eid}{\printfield{volume}\setunit*{\addnbspace}\printfield{number}\setunit{\addcomma\space}\printfield{eid}}
\DeclareFieldFormat[article]{volume}{\textbf{#1}\space}
\DeclareFieldFormat[article]{number}{\mkbibparens{#1}}
\DeclareFieldFormat{journaltitle}{#1,}
\DeclareFieldFormat[thesis]{title}{\mkbibemph{#1}\addperiod}
\DeclareFieldFormat[article, unpublished, thesis, incollection, inproceedings]{title}{\mkbibemph{#1},}
\DeclareFieldFormat[book]{title}{\mkbibemph{#1}\addperiod}
\DeclareFieldFormat[unpublished]{howpublished}{#1, }
\DeclareFieldFormat{pages}{#1}
\DeclareFieldFormat[article]{series}{Ser.~#1\addcomma}
\DeclareBibliographyDriver{misc}{%
  \usebibmacro{bibindex}%
  \usebibmacro{begentry}%
  \usebibmacro{author/editor+others/translator+others}%
  \setunit{\labelnamepunct}\newblock
  \usebibmacro{title}%
  \newunit
  \printlist{language}%
  \newunit\newblock
  \usebibmacro{byauthor}%
  \newunit\newblock
  \usebibmacro{byeditor+others}%
  \newunit\newblock
  \printfield{howpublished}%
  \newunit\newblock
  \printfield{type}%
  \newunit
  \printfield{version}%
  \newunit\newblock
  \usebibmacro{organization+location+date}%
  \newunit\newblock
  \usebibmacro{doi+eprint+url}%
  \newunit\newblock
  \usebibmacro{addendum+pubstate}%
  \setunit{\bibpagerefpunct}\newblock
  \usebibmacro{pageref}%
  \newunit\newblock
  \iftoggle{bbx:related}
    {\usebibmacro{related:init}%
     \usebibmacro{related}}
    {}%
  \newunit\newblock
  \printfield{note}%
  \usebibmacro{finentry}%
}

\setlist{topsep=3pt,itemsep=3pt}



\makeatletter
\renewcommand\fs@ruled{%
  \def\@fs@cfont{\bfseries}%
  \let\@fs@capt\floatc@ruled
  \def\@fs@pre{%
    \hrule height .8pt depth 0pt
    \kern5pt
  }%
  \def\@fs@mid{%
    \kern5pt
    \hrule
    \kern5pt
  }%
  \def\@fs@post{%
    \kern5pt
    \hrule
    \relax
  }%
  \let\@fs@iftopcapt\iftrue
}
\makeatother

\title{\sffamily Random Independent Sets in triangle-free and \\linear berge-$C_4$-free Hypergraphs}
\date{}

\author[J. Yu]{Jing~Yu}
\author[J. Zhang]{Junchi~Zhang}

\begin{document}
\begin{abstract}
We study probability distributions on independent sets of uniform hypergraphs with local constraints.
For a positive vertex-weight function $w$ on an $(r+1)$-uniform hypergraph, define the weighted degree of vertex $v$ by
\[
d_w(v):=
\sum_{e\ni v}
\left(
\prod_{u\in e\setminus\{v\}}\dfrac{w(u)}{w(v)}
\right)^{1/r}.
\]
For each fixed $r\geq1$, our first result gives, in every triangle-free
$(r+1)$-uniform hypergraph (without a linearity assumption), a random independent set $I$ satisfying
\[
\mathbb P(v\in I)\geq
\left(
\frac{r}{r+1}(r(r+1))^{-1/r}+o_r(1)
\right)
\left(\frac{\log d_w(v)}{d_w(v)}\right)^{1/r}, \qquad d_w(v)\to+\infty.
\]
The same result also implies that every $d$-degenerate triangle-free $(r+1)$-uniform hypergraph satisfies
\[
\chi_f(G)\leq
\left(
\dfrac{r+1}{r}(r(r+1))^{1/r}+o_r(1)
\right)
\left(\dfrac{d}{\log d}\right)^{1/r}, \qquad d\to+\infty.
\]
Our second result concerns linear Berge-$C_4$-free hypergraphs which allow Berge triangles. 
In this setting, for every sufficiently large threshold $D$, there exists a random independent set $I$ such that, uniformly over all vertices with $d_w(v)\geq D$,
\[
\mathbb P(v\in I)\geq
(1-o_r(1))
\left(\frac{\log d_w(v)}{r\,d_w(v)}\right)^{1/r}, \qquad D\to+\infty.
\]
It also yields
\[
\chi_f(G)\leq
(1+o_r(1))
\left(\frac{rd}{\log d}\right)^{1/r}
\]
for $d$-degenerate linear Berge-$C_4$-free $(r+1)$-uniform hypergraphs. 

Both results are based on the ordered random pick process originated from Martinsson and Steiner. By choosing different selection functions and iteration, we prove the output of the random process gives the required random independent set.
\end{abstract}

\maketitle
\authorfootnote{(JY) Shanghai Center for Mathematical Sciences, Fudan University, Shanghai, China; 
e-mail: \texttt{jyu@fudan.edu.cn}. Partially supported by the National Natural Science Foundation of China grants 12371343 and 12525110 (PI: Hehui Wu).\\[2pt]
(JZ) Shanghai Center for Mathematical Sciences, Fudan University, Shanghai, China; 
e-mail: \texttt{jczhang24@m.fudan.edu.cn}.}

\section{Introduction}
How large an independent set must a sparse graph contain? A stronger question is, can we find a random independent set in which every vertex has a prescribed lower bound on its probability of appearing?
The first question is classical in extremal combinatorics.
For triangle-free graphs it is closely connected with Ramsey theory: $R(3,k) \le n$ means that every triangle-free graph on at least $n$ vertices contains an independent set of size $k$.
Ajtai, Koml\'os, and Szemer\'edi~\cite{Ajtai1980ramsey} proved that an $n$-vertex triangle-free graph of average degree $d$ has an independent set of size at least $c n\log d/ d$ for an absolute constant $c>0$. 
Writing $\alpha(G)$ for the size of a largest independent set, 
Shearer~\cite{Shearer1983triangle} strengthened this to
\[
 \alpha(G)\geq(1-o(1))\frac{n\log d}{d}.
\]
This lower bound on the independence number is referred to as \emph{Shearer's bound}.  His subsequent degree-sequence bound~\cite{Shearer1983triangleII} replaces the average-degree term by a sum of contributions from the vertex degrees.

Coloring provides a stronger route to large independent sets.
Write $\chi(G)$ and $\chi_\ell(G)$ for the \emph{chromatic} and \emph{list chromatic numbers}.
Johansson~\cite{Johansson} established $\chi_\ell(G)=O(\Delta/\log\Delta)$ for triangle-free graphs of maximum degree $\Delta$. Molloy~\cite{Molloy19list} subsequently improved the leading coefficient to $1+o(1)$.

The \emph{fractional chromatic number} provides a convenient setting for vertex-dependent inclusion guarantees.
For a graph or hypergraph $G=(V,E)$, the \emphd{fractional chromatic number} of $G$, denoted by $\chi_f(G)$, is the minimum real number $t\geq 1$ for which there is a probability distribution $\mathcal D$ on independent sets of $G$ such that
\[
 \P_{I \sim \mathcal D}(v\in I)\geq\frac1t\qquad\text{for every }v\in V.
\]
Taking the expected size of such an independent set and choosing a color class uniformly from a proper coloring gives
\[
 \alpha(G)\geq\frac{|V|}{\chi_f(G)}\geq\frac{|V|}{\chi(G)}.
\]

More generally, one may prescribe a \emphd{demand} $h(v)\in[0,1]$ for each vertex, meaning a required lower bound on its inclusion probability, and seek a distribution on independent sets satisfying $\P(v\in I)\geq h(v)$ for every $v \in V$. 
Unlike the common requirement $1/t$ above, these demands may vary from vertex to vertex. This is the distributional formulation of \emphd{fractional coloring with local demands} studied by  Kelly and Postle~\cite{Kelly2024frac}.
They conjectured that Shearer's graph bound admits a local fractional strengthening, with the inclusion probability at $v$ asymptotic to $\log d_G(v)/d_G(v)$.
Martinsson and Steiner~\cite{Martinsson2025RandomInd} proved this conjecture in a more general weighted version.

Another strengthening replaces maximum degree by degeneracy.
A graph is \emphd{$d$-degenerate} if its vertices can be ordered so that each has at most $d$ earlier neighbors.
Greedy coloring gives $\chi(G)\leq d+1$, and triangle-freeness alone does not permit a logarithmic improvement of this ordinary coloring bound; see~\cite{shengtong26degenerate} for the classical constructions and the role of additional size restrictions. However, for fractional coloring, Martinsson and Steiner~\cite{Martinsson2025RandomInd} proved that every $d$-degenerate triangle-free graph $G$ satisfies
\[
 \chi_f(G)\leq(4+o(1))\frac d{\log d}.
\]
Recently, Allen, Dhawan and Noel~\cite{allen2026sharpboundsfractionalchromatic} obtained the sharp leading coefficient 1 for $C_4$-free $d$-degenerate graphs. 

Natural questions arise in the hypergraph setting.
Throughout, a \emphd{hypergraph} is a finite pair $G=(V,E)$ with
\[
 E\subseteq\{e\subseteq V:|e|\geq2\},
\]
and a set is \emphd{independent} if it contains no edge.
The hypergraph $G$ is called \emphd{$k$-uniform} if every edge has size $k$, and \emphd{linear} if any two distinct edges intersect in at most one vertex.
The \emphd{rank} of a hypergraph is the maximum size of its edges.
For $\ell\geq2$, a \emphd{Berge cycle} of length $\ell$ consists of distinct edges $e_1,\ldots,e_\ell$ and distinct vertices $x_1,\ldots,x_\ell$ such that $\{x_i,x_{i+1}\}\subseteq e_i$, with indices taken modulo $\ell$. 
The vertices $x_1,\ldots,x_\ell$ are called the \emphd{core vertices} of the cycle.
A \emphd{Berge triangle} is a Berge cycle of length three.
Thus linearity is equivalent to the absence of Berge cycles of length two.
A hypergraph is \emphd{uncrowded} if it has no Berge cycles of lengths two, three, or four.  A hypergraph is \emphd{$d$-degenerate} if it admits a vertex order in which each vertex is the largest vertex of at most $d$ edges.
Equivalently, every nonempty vertex-induced subhypergraph has a vertex contained in at most $d$ of its edges.

For fixed $r\geq 2$, the counterpart of the logarithmic graph bound is an independent set of size
\[
 \Omega_r\!\left(n\left(\frac{\log\Delta}{\Delta}\right)^{1/r}\right)
\]
in an $(r+1)$-uniform hypergraph of maximum degree $\Delta$, under suitable restrictions on short cycles.
Ajtai, Koml\'os, Pintz, Spencer, and Szemer\'edi~\cite{Ajtai1982uncrowded} established this bound for uncrowded hypergraphs.
Duke, Lefmann, and R\"odl~\cite{Lefmann95uncrowded} extended it to linear hypergraphs.
Frieze and Mubayi~\cite{Mubayi13color} proved the corresponding chromatic bound
\[
 \chi(G)\leq O_r\!\left(\left(\frac{\Delta}{\log\Delta}\right)^{1/r}\right)
\]
for linear hypergraphs.
Degree-sequence independence estimates in the linear Berge-triangle-free setting were obtained by Borowiecki, Gentner, L\"owenstein, and Rautenbach~\cite{Piotr16independence}.

We also use the following notion of triangle-freeness, as in~\cite{Mubayi15list,Postle25chrom}: a \emphd{triangle} consists of three distinct edges $e,f,g$ and three distinct vertices $u,v,z$ such that 
\[
\set{u,v}\subseteq e,
    \qquad
    \set{v,z}\subseteq f,
    \qquad
    \set{z,u}\subseteq g,
   \qquad
  \set{u,v,z}\cap e\cap f\cap g=\0.
\]
A hypergraph is \emphd{triangle-free} if it contains no such triangle.
A triangle is a Berge $C_3$, and in linear hypergraphs these two notions agree.
Cooper and Mubayi~\cite{Mubayi15list} proved a logarithmic list-coloring bound for triangle-free hypergraphs of rank $3$, and Li and Postle~\cite{Postle25chrom} extended it to arbitrary rank. 
In particular, for triangle-free $(r+1)$-uniform hypergraphs their result gives
\[
 \chi_\ell(G)\leq O_r\!\left(\left(\frac{\Delta}{\log\Delta}\right)^{1/r}\right).
\]

Leading constants have also received attention. For uncrowded $(r+1)$-uniform hypergraphs, Iliopoulos~\cite{Ili21} obtained the list-coloring upper bound $(1+o(1))r(\Delta/\log\Delta)^{1/r}$. The smaller bound
\[
 \chi_\ell(G)\leq(1+o_r(1))\left(\frac{r\Delta}{\log\Delta}\right)^{1/r}
\]
is proved in Yu and Zhang~\cite{yu2026listcoloringnumberuncrowded} and in the second version of manuscript by Dhawan, Methuku, and Vo~\cite{dhawan2026listcoloringuncrowdedhypergraphs} on Sept. 7th.
Our earlier work~\cite{yu2026hypergraphindependenceboundsmaximum} transfers the corresponding independence-number bound from maximum degree to average degree.
For $d$-degenerate hypergraphs, Dhawan~\cite{dhawan2026fractionalcoloringentropy} used entropy to prove that $(r+1)$-uniform $d$-degenerate hypergraphs with no Berge cycles of lengths two or three satisfy
\[
 \chi_f(G)\leq
 \left(\frac{r+1}{r}\bigl((r+1)r^2\bigr)^{1/r}+o_r(1)\right)
 \left(\frac d{\log d}\right)^{1/r},
\]
and also obtained local-demand versions. 
He asked  \cite[Conjecture~5.4]{dhawan2026fractionalcoloringentropy} whether the same order of magnitude holds for triangle-free uniform hypergraphs without a linearity assumption.

Our first result resolves~\cite[Conjecture~5.4]{dhawan2026fractionalcoloringentropy}.
For $r\geq2$, its degeneracy consequence also improves the leading coefficient in~\cite[Theorem~1.4]{dhawan2026fractionalcoloringentropy} by a factor of $r^{1/r}$, while removing the linearity assumption.
Our second result attains the coefficient proposed in~\cite[Conjecture~5.3]{dhawan2026fractionalcoloringentropy} on the uncrowded class, and extends the same bound to linear Berge-$C_4$-free hypergraphs that may contain Berge triangles.

\subsection{Main results}

We should first introduce some basic notions.
Given a hypergraph $G=(V,E)$ with an order $\prec$ on $V$, the \emphd{head} and \emphd{neck} of an edge $e\in E$, denoted by $m(e)$ and $m^*(e)$, are its largest and second largest vertices.
Put
\[
e^-\defeq e\setminus \{m(e)\},\;e\in E, \qquad \B(v)\defeq \{e\in E\;:\; m(e)=v\},\; v\in V.
\]
We call $e^-$ the \emphd{tail} of $e$, and the edges in $\B(v)$ \emphd{backward edges} of $v$.
For $x\prec v$, define 
\[
 E(x,v)=\{e\in E\;:\; m^*(e)=x,\; m(e)=v \}.
\]
The order $\prec$ is called \emphd{good} if for any $f\in E$ and any vertex $x\notin f$,
\begin{equation}\label{eq:good}
|\{  v\in f \;:\; x\prec v,\;  E(x,v)\neq\emptyset\}|\leq 1 .
\end{equation}
Thus in a good order, a vertex outside a fixed edge can potentially affect at most one later vertex of that edge. 
Every order of a triangle-free hypergraph is good, as verified at the start of Section \ref{sec:triangle}.

A \emphd{weight function} is a map $w:V\to(0,\infty)$ such that for any $S\subseteq V$,
\[
w(S)=\sum_{v\in S}w(v).
\]
In particular, $w(\varnothing)=0$; no upper bound on $w(S)$ is imposed.
Separately, an \emphd{activity vector} is a map $a:V\to(0,1)$ specifying the initial masses in the random process. No additive extension of $a$ to subsets is assumed or used. We write $a_i(u)$ for the \emphd{evolving mass} of an unprocessed vertex $u$ after step $i$, with $a_0(u)=a(u)$.
For an activity vector $a$ and a fixed order, define
\[
\Lambda_a(v):=\sum_{e\in\B(v)}a(m^*(e))
 \prod_{u\in e\setminus\{m^*(e),v\}}(1-e^{-a(u)}),\qquad
 \Gamma_a(v):=\sum_{e\in\B(v)}\prod_{u\in e^-}a(u).
\]
These quantities record two versions of the contribution from the backward edges, corresponding to the selection rules used in the proofs of the two main theorems. Note that for rank at most $R$, always 
\[
\Lambda_a(v)\le\Gamma_a(v),\qquad \Lambda_a(v) = (1+o_R(1))\Gamma_a(v)\quad\text{uniformly in }v\quad \text{as }\max_{u\in V} a(u)\to 0.
\]

%Our two main theorems give distributions on independent set such that the probability of each vertex to be contained is related to the weight of itself and its neighborhood.
The first theorem focuses on triangle-free graphs without uniformity assumption.

\begin{theo}\label{thm:maintriangle}
    Let a hypergraph $G=(V,E)$ have a good total order $\prec$ on $V$. For every $a: V \to (0,1)$, there exists a distribution $\D$ on independent sets of $G$ such that for any $v\in V$,
    \[
    \P_{I\sim \D}(v\in I)\geq (1-e^{-a(v)})\left( 1+\dfrac{\Lambda_a(v)}{\log a(v)} \right).
    \]
    In particular, the conclusion holds for all triangle-free hypergraphs with any order.
\end{theo}

The second result concerns linear Berge-$C_4$-free hypergraphs.

\begin{theo}\label{thm:mainC4}
For every integer $r\geq1$ and every $\epsilon\in(0,1)$, there exists $\delta=\delta(r,\epsilon)>0$ such that the following holds.
Let $G=(V,E)$ be a linear Berge-$C_4$-free $(r+1)$-uniform hypergraph.
For every total order $\prec$ on $V$ and every $a:V\to(0,1)$, there exists a distribution $\mathcal D$ on the independent sets of $G$ such that, simultaneously for every $v\in V$,
\begin{equation*}
    \P_{I\sim\mathcal D}(v\in I)
    \geq
    (1-\epsilon)
    \min\left\{
        a(v),\,
        \delta e^{-\Gamma_a(v)}
    \right\}.
\end{equation*}
\end{theo}

\subsection{Applications}

Given a hypergraph $G=(V,E)$, the \emphd{degree} and \emphd{neighborhood} of a vertex $v\in V$ are defined by
\[
 d_G(v)\defeq\bigl|\set{e\in E:v\in e}\bigr|, \qquad 
 N_G(v)\defeq  \set{u\in V\setminus\set{v}\;:\;\exists\; e\in E,\; \set{u,v}\subseteq e}. 
\]
When $G$ is $(r+1)$-uniform, 
define the \emphd{weighted degree} of a vertex $v$ under a weight function $w$ by
\[
d_w(v) \defeq \sum_{e\ni v}
    \left(\prod_{u\in e\setminus\set{v}}
        \frac{w(u)}{w(v)}\right)^{1/r}.
\]
Taking $w \equiv 1$ gives $d_w =d_G$ for every $r$.
For $r\geq 1$ and $x \ge 0$, set
\[
\ell_r(x)\defeq \left( \dfrac{\log(x+3)}{r(r+1)(x+3)} \right)^{1/r},\qquad F_r(x)\defeq (1-e^{-\ell_r(x)})\left( 1+\dfrac{x\ell_r(x)^r}{\log \ell_r(x)}  \right).
\]
As we will see in \Cref{lem:Frx}, $F_r$ is positive on $[0,\infty)$ and
\[
F_r(x)=\left( \dfrac{r}{r+1} (r(r+1))^{-1/r}+o_r(1) \right) \left(\dfrac{\log x}{x}\right)^{1/r}.
\]
The first applications are hypergraph versions of the local Shearer bound in \cite[Theorem 1.2]{Martinsson2025RandomInd} with a constant-factor loss. 
\begin{corl}[Local Shearer bound for triangle-free hypergraphs]\label{cor:localshearerbound}
    Let  $r\geq 1$ be an integer and let $G=(V,E)$ be a triangle-free $(r+1)$-uniform hypergraph. For every weight function $w: V \to (0,\infty)$, there exists a random independent set $I$ such that
    \[
    \P(v\in I)\geq F_r(d_w(v))
    \]
    for every $v\in V$.
\end{corl}

\begin{corl}[Local Shearer bound for uniform linear Berge-$C_4$-free hypergraphs]
\label{thm:c4-weighted-local}
    For every integer $r\geq 1$ and every $\epsilon\in(0,1)$, there exists $D=D(r,\epsilon) > 1$ such that the following holds.
    Let $G=(V,E)$ be a linear Berge-$C_4$-free $(r+1)$-uniform hypergraph with weight function $w$.
    There is one random independent set $I$ such that every vertex $v$ with $d_w(v)\geq D$ satisfies
    \begin{equation*}
    \P(v\in I)
    \geq
    (1-\epsilon)
    \left(
        \frac{\log d_w(v)}
             {r\,d_w(v)}
    \right)^{1/r}.
    \end{equation*}
\end{corl}

The next applications bound the fractional chromatic number for $d$-degenerate hypergraphs.

\begin{corl}\label{cor:fractriangle}
    For every fixed integer $r\geq1$,  every $d$-degenerate triangle-free $(r+1)$-uniform hypergraph $G$ satisfies 
    \[
    \chi_f(G)\leq \left( \dfrac{r+1}{r} (r(r+1))^{1/r}+o_r(1) \right)\left( \dfrac{d}{\log d} \right)^{1/r}.
    \]
\end{corl}

\begin{corl}\label{cor:c4-fractional}
For every fixed integer $r\geq1$, every linear
Berge-$C_4$-free, $d$-degenerate $(r+1)$-uniform
hypergraph $G$ satisfies
\begin{equation*}
    \chi_f(G)
    \leq
    \bigl(1+o_r(1)\bigr)
    \left(\frac{rd}{\log d}\right)^{1/r}.
\end{equation*}
\end{corl}

For $r=1$, \Cref{cor:fractriangle} recovers \cite[Theorem 1.7]{Martinsson2025RandomInd}, and \Cref{cor:c4-fractional} recovers the upper bound in \cite[Theorem 1.3]{allen2026sharpboundsfractionalchromatic}. For $r \ge 2$, \Cref{cor:fractriangle} improves the coefficient in \cite[Theorem 1.4]{dhawan2026fractionalcoloringentropy} by a factor of $r^{1/r}$ and does not require linearity.  The leading constant in \Cref{cor:c4-fractional} matches that in the
maximum-degree list-coloring bound in~\cite{yu2026listcoloringnumberuncrowded} and \cite{dhawan2026listcoloringuncrowdedhypergraphs}; here, however, the parameter is degeneracy and the conclusion is for fractional coloring.

\subsection{Proof overview}

Both arguments are based on the ordered random pick process in \cite{Martinsson2025RandomInd}, extended to hypergraphs. When the current vertex is selected with probability $q_i=p(a_{i-1}(v))$ with a selection function $p$, an affected future mass $B$ is set to zero; on the other branch it becomes $B/(1-q_i)$. Thus its conditional expectation remains $B$.

For \Cref{thm:maintriangle}, we use $p(t)=1-e^{-t}$ and the fixed-target modification from~\cite[Section 4]{Martinsson2025RandomInd}. 
In the graph case, one only needs to track the mass of an earlier neighbor of the target. In the hypergraph setting, one must also record whether the other earlier vertices of a backward edge have been selected. We encode this information in an auxiliary product. The good-order condition ensures that an outside step affects at most one factor of this product, leading to the required estimate.

For \Cref{thm:mainC4}, we use the capped selection function $p_\rho(t)=\min\{t,\rho\}$ and a second-moment argument related to~\cite{allen2026sharpboundsfractionalchromatic}. The cap gives uniform control on the compensation factors on vertex weights while preserving the first moment of each vertex weight. To bound the second moment, we track products associated with backward edges and iterate the resulting one-step bounds over their subsets. Linearity and Berge-$C_4$-freeness ensure that the steps where several relevant factors can interact are sufficiently sparse, allowing the accumulated loss to be bounded.

\subsection{Organization}
\Cref{sec:prelim} gives the common process and auxiliary lemmas. \Cref{sec:triangle} and \Cref{sec:C4} prove the two main theorems. \Cref{sec:corollaries} proves the applications.

\subsection*{Note added}
In version~4 of~\cite{dhawan2026fractionalcoloringentropy}, Dhawan announces forthcoming joint work proving the conjectured leading-coefficient degeneracy bound, together with a local-demand version, for uncrowded hypergraphs.
Our second theorem has a different scope: it applies to linear Berge-$C_4$-free hypergraphs, and hence allows Berge triangles.
At the time of writing, the cited version contains only this announcement.

\section{The ordered random pick process and auxiliary lemmas}\label{sec:prelim}
All logarithms are natural, and an empty product is one.
Write $V_i:=\{v_1,\ldots,v_i\}$ for the first $i$ vertices in the chosen order, with $V_0= \emptyset$.
The following \emphd{ordered random pick process} will be used both in the proofs of \Cref{thm:maintriangle} and \Cref{thm:mainC4}. 
\if0
Given a hypergraph $G=(V,E)$ with order $v_1\prec\cdots\prec v_n$, initial activities $a\in (0,1)^V$ and a measurable selection function $p\colon [0,\infty)\to [0,1)$ with $p(0)=0$. Initialize $a_0(u)=a(u)$ for every $u \in V$ and $I_0=\0$. 
Before step $i$, put
\[
 q_i:=p(a_{i-1}(v_i)),\qquad
 L_i(v_j):=\mathbf{1}_{\{\exists e\in E(v_i,v_j):\ e\setminus\{v_i,v_j\}\subseteq I_{i-1}\}}
 \quad(j>i).
\]
Then perform one of the following two updates, conditionally on the previous choices.
\begin{enumerate}
    \item With probability $q_i$, set $I_i=I_{i-1}\cup\{v_i\}$. For $j>i$, set $a_i(v_j)=0$ if $L_i(v_j)=1$, and otherwise keep $a_i(v_j)=a_{i-1}(v_j)$.
\item With probability $1-q_i$, set $I_i=I_{i-1}$. For $j>i$, set
\[
 a_i(v_j)=\frac{a_{i-1}(v_j)}{1-q_i}
\]
if $L_i(v_j)=1$, and otherwise keep $a_i(v_j)=a_{i-1}(v_j)$.
\end{enumerate}
\fi
\begin{algorithm}[!htbp]
\caption{Ordered random pick process}
\label{alg:ordered-process}

\begingroup
\raggedright

\setlength{\abovedisplayskip}{5pt}
\setlength{\belowdisplayskip}{5pt}
\setlength{\abovedisplayshortskip}{3pt}
\setlength{\belowdisplayshortskip}{4pt}

\noindent\textbf{Input:}
A hypergraph $G=(V,E)$ with an order
$v_1\prec\cdots\prec v_n$, an initial activity vector
$a\in(0,1)^V$, and a measurable selection function
$p:[0,\infty)\to[0,1)$ with $p(0)=0$.

\smallskip

\noindent\textbf{Output:}
A random independent set $I_n$.

\begin{enumerate}
\item Set
\[
I_0=\varnothing,
\qquad
a_0(u)=a(u)\quad(u\in V).
\]

\item For $i=1,\ldots,n$, put
\[
q_i:=p(a_{i-1}(v_i)),
\]
and, for every $j>i$, define
\[
L_i(v_j):=\mathbf 1_{\{\exists e\in E(v_i,v_j)\,:\,e\setminus\{v_i,v_j\}\subseteq I_{i-1}
\}}.
\]
For every $j>i$ with $L_i(v_j)=0$, set
\[
a_i(v_j)=a_{i-1}(v_j).
\]
Select $v_i$ with probability $q_i$, conditionally on the previous choices.
If $v_i$ is selected, set
\[
I_i=I_{i-1}\cup\{v_i\},
\]
and, for every $j>i$ with $L_i(v_j)=1$, set
\[
a_i(v_j)=0.
\]
If $v_i$ is not selected, set
\[
I_i=I_{i-1},
\]
and, for every $j>i$ with $L_i(v_j)=1$, set
\[
a_i(v_j)
=
\frac{a_{i-1}(v_j)}{1-q_i}.
\]

\item Return $I_n$.

\end{enumerate}

\endgroup
\end{algorithm}

A future vertex is \emphd{affected} at step $i$ when $L_i(v_j)=1$; its mass is updated once, even if several edges witness this event. The evolving masses may exceed one. Only their images under $p$ are selection probabilities. 
Let $\F_i$ denote the information revealed by the first $i$ steps, that is, the $\sigma$-algebra generated by their branch choices, with $\F_0$ trivial.
By construction, both $q_i$ and $L_i(v_j)$ are determined by the history up to step $i-1$, and hence are $\F_{i-1}$-measurable.

\begin{prop}\label{prop:process}
    The ordered random pick process always outputs an independent set. Consequently, its output defines a probability distribution on the independent sets of $G$.
    Moreover, for any $1\leq i<j\leq n$,
    \[
    \E[  a_i(v_j) \mid  \F_{i-1}  ]=a_{i-1}(v_j), \quad \E[  a_i(v_j)]=a(v_j).
    \]
    In particular, for $v=v_k$,
 \[
  \E a_{k-1}(v)=a(v),\qquad
  \P(v\in I_n)=\E p(a_{k-1}(v)).
 \]
\end{prop}
\begin{proof}
If $L_i(v_j)=0$, then $a_i(v_j)=a_{i-1}(v_j)$. If $L_i(v_j)=1$, then, conditioned on $\F_{i-1}$,
\[
a_i(v_j)=
\begin{cases}
0,&\text{with probability }q_i,\\[1mm]
a_{i-1}(v_j)/(1-q_i),&\text{with probability }1-q_i.
\end{cases}
\]
Hence $\E[a_i(v_j)\mid\F_{i-1}]=a_{i-1}(v_j)$, and iteration gives $\E a_i(v_j)=a(v_j)$. In particular, if $v=v_k$, then $\E a_{k-1}(v)=a(v)$, while conditioning on $\F_{k-1}$ gives
\[
\P(v\in I_n)=\E p(a_{k-1}(v)).
\]
Finally, suppose that some edge $e$ were contained in $I_n$. When its neck $m^*(e)$ is processed, all earlier vertices of $e$ have already been selected, so selecting the neck sets the mass of the head $m(e)$ to zero. Since $p(0)=0$, the head cannot subsequently be selected, a contradiction.
\end{proof}

We shall also use the following inequality in the proof of \Cref{thm:maintriangle}.

\begin{lemma}\label{ineq:e-t}
    For $\alpha \in (0,1)$ and $t\geq 0$, we have
    \[
    e^{-t}(1-e^{-\alpha e^t})\geq (1-e^{-\alpha})\left(1+\dfrac{t}{\log \alpha}\right),
    \]
    and the right hand side is decreasing in $t$.
\end{lemma}
\begin{proof}
    We may assume $t\leq -\log \alpha$, since otherwise the right hand side is nonpositive.
    Set
    \[
    f_\alpha(t)\defeq e^{-t}\left(1-e^{-\alpha e^t}\right),
    \qquad
    s \defeq \alpha e^t.
    \]
    Direct differentiation gives
    \begin{equation}\label{eq:second-derivative}
    f_\alpha''(t) = e^{-t-s}\left(e^s-1-s-s^2\right).
    \end{equation}
    Since $0\leq t\leq -\log \alpha$, we have $0<s\leq1$ and $e^s\leq1+s+s^2$.
    Thus $f_\alpha$ is concave on this interval.
    Since $f_\alpha(0)=1-e^{-\alpha}$ and $f_\alpha(-\log \alpha)\geq 0$, the concavity yields
    \[
    f_\alpha(t)\geq (1+t/\log \alpha)f_\alpha(0)+(-t/\log \alpha)f_\alpha(-\log \alpha)\geq (1+t/\log \alpha)f_\alpha(0)=(1-e^{-\alpha})\left(1+\dfrac{t}{\log \alpha}\right).
    \]
    The final assertion follows from $\log \alpha < 0$.
\end{proof}

To prove \Cref{cor:fractriangle}, we need the following properties of $\ell_r(x)$ and $F_r(x)$.
\begin{lemma}\label{lem:Frx}
For every integer $r\geq1$ and every $x\geq0$, we have $0<\ell_r(x)<1$ and $F_r(x)>0$. Moreover,
\[
 F_r(x)=\left(\frac r{r+1}(r(r+1))^{-1/r}+o_r(1)\right)
 \left(\frac{\log x}{x}\right)^{1/r}\qquad(x\to\infty).
\]
\end{lemma}

\begin{proof}
Put $D=x+3$ and $\alpha=\ell_r(x)$, so that
$\alpha^r=\log D/(r(r+1)D)\in(0,1)$. Since $D\geq3$,
\[
 \log D\leq\sqrt D\leq D^{r/(r+1)},\qquad
 \log(1/\alpha)=\frac1r\log\frac{r(r+1)D}{\log D}
 \geq\frac1r\log\frac{D}{D^{r/(r+1)}}
 =\frac{\log D}{r(r+1)}.
\]
Consequently,
\[
 \frac{x\alpha^r}{\log(1/\alpha)}
 \le \dfrac{x\log D/(r(r+1)D)}{\log D/\left(r(r+1)\right)}
 =\frac xD<1.
\]
Hence 
\[F_r(x) = (1-e^{-\alpha})\left(1+\frac{x\alpha^r}{\log \alpha}\right) =(1-e^{-\alpha})\left(1-\frac{x\alpha^r}{\log (1/\alpha)}\right) > 0. \]
As $x\to\infty$,
\[
\ell_r(x)^r
=
\frac{\log(x+3)}{r(r+1)(x+3)}
=
(1+o_r(1))
\frac{\log x}{r(r+1)x},
\]
and hence
\[
\ell_r(x)
=
(1+o_r(1))(r(r+1))^{-1/r}
\left(\frac{\log x}{x}\right)^{1/r}.
\]
In particular, $\ell_r(x)\to0$, so
\[
1-e^{-\ell_r(x)}
=
(1+o_r(1))\ell_r(x).
\]
Moreover,
\[
x\ell_r(x)^r
=
(1+o_r(1))
\frac{\log x}{r(r+1)},
\]
while
\[
\log\ell_r(x)
=
\frac1r
\Bigl(
\log\log(x+3)-\log(r(r+1))-\log(x+3)
\Bigr)
=
-\frac1r\log x+o_r(\log x).
\]
Therefore
\[
\frac{x\ell_r(x)^r}{\log\ell_r(x)}
=
-\frac1{r+1}+o_r(1).
\]
Substituting these estimates into the definition of $F_r$ gives
\[
F_r(x)
=
\left(
\frac{r}{r+1}(r(r+1))^{-1/r}
+o_r(1)
\right)
\left(\frac{\log x}{x}\right)^{1/r}.\qedhere
\]
\end{proof}

Fix an ordered linear hypergraph, a target $v$, and $e\in\B(v)$. For $x\prec v$, put
\begin{equation}\label{eq:Ne}
 N_e(x):=\{u\in e^-:x\prec u,\ E(x,u)\ne\varnothing\}.
\end{equation}
Thus $N_e(x)$ records the tail vertices that $x$ can potentially affect in the ordered random pick process. If $x\in e^-$, then $N_e(x)=\varnothing$: a witnessing edge would have two vertices in common with $e$. Call $(x,e)$ \emphd{coupled}, relative to the target $v=m(e)$, if either
\begin{equation}\label{eq:coupled}
 |N_e(x)|\geq2,
 \quad\text{or}\quad
 N_e(x)\ne\varnothing\text{ and }x\in f^-
 \text{ for some }f\in\B(v).
\end{equation}
This is a condition on the ordered hypergraph, independent of the random history. The first alternative records that $x$ may simultaneously affect at least two tail vertices of $e$. The second records that $x$ may affect $e$ from outside while $x$ itself lies in the tail of another backward edge of the target $v$. 
%which also produce many triangles as in the proof of the following lemma.

\begin{lemma}\label{lem:c4-separation}
Given a linear Berge-$C_4$-free $(r+1)$-uniform hypergraph $G=(V,E)$ with any total order $\prec$ on $V$, where $r \ge 1$, and fix $v \in V$. Then the following statements hold.
\begin{enumerate}[label=\textnormal{(\arabic*)}]
    \item\label{lem:c4-separation1} For every $x\prec v$, at most one edge $e\in\mathcal B(v)$ has $N_e(x)\ne\0$.
    \item \label{lem:c4-separation2} Every edge $e\in\mathcal B(v)$ is coupled with at most $T_r\defeq r\binom{r+1}{2}$ vertices.
\end{enumerate}
\end{lemma}
\begin{proof}
If $u\in N_e(x)$, choose $g\in E(x,u)$. The heads of $g$ and $e$ differ, and linearity implies $x\notin e$.

For the first statement, suppose that $u\in N_e(x)$ and $z\in N_f(x)$ for distinct $e,f\in\B(v)$. Choose $g\in E(x,u)$ and $h\in E(x,z)$. Since $e\cap f=\{v\}$ and $x\notin e\cup f$, the vertices $v,u,x,z$ are distinct. The edges $g,h$ have distinct heads earlier than $v$, so $e,g,h,f$ are distinct. Then
\[
 v,e,u,g,x,h,z,f,v
\]
form a Berge $C_4$, a contradiction.

For the second statement, fix $e\in\mathcal B(v)$. We first show that each vertex $x$ coupled with $e$ gives rise to a Berge triangle containing $e$.
If $|N_e(x)|\geq 2$, choose distinct $u,z\in N_e(x)$ and witnessing edges $g\in E(x,u)$ and $h\in E(x,z)$. Then $e,g,h$ form a Berge triangle with core vertices $u,x,z$.
Otherwise, $N_e(x)\neq\varnothing$ and $x\in f^-$ for some $f\in\mathcal B(v)$. Choose $u\in N_e(x)$ and $g\in E(x,u)$. Then $e,f,g$ form a Berge triangle with core vertices $u,v,x$.
In either case, $x$ is the unique core vertex outside $e$. Consequently, distinct coupled vertices give rise to distinct Berge triangles containing $e$.

It remains to show that $e$ belongs to at most $T_r=r\binom{r+1}{2}$  Berge triangles.
Every Berge triangle containing $e$ has exactly two core vertices in $e$. Fix such a pair $\{s,t\}\subseteq e$. Let $H$ be the bipartite graph whose two vertex classes consist of the edges other than $e$ containing $s$ and $t$, respectively. Two such edges are adjacent in $H$ if they intersect.
By linearity, each edge of $H$ corresponds to a unique Berge triangle containing $e$ with core vertices $s,t$ on $e$.
We claim that $H$ has matching number at most one. 
Otherwise, let $f_1g_1$ and $f_2g_2$ be two disjoint edges of $H$. By linearity, write
\[
\set{x_1}=f_1\cap g_1,
\qquad
\set{x_2}=f_2\cap g_2,
\]
where $x_1,x_2 \notin e$.
Moreover, $x_1\neq x_2$, since otherwise $f_1$ and $f_2$ would share both $s$ and $x_1$. 
Hence
\[
s,f_1,x_1,g_1,t,g_2,x_2,f_2,s
\]
form a Berge $C_4$, a contradiction.
Thus, if $H$ is nonempty, all its edges form a star. Its degree is at most $r$. Indeed, if its center corresponds to an edge $f$ containing $s$, then distinct neighbors meet $f$ at distinct vertices of $f\setminus\{s\}$; otherwise two such neighbors would share two vertices, contradicting linearity. The case where the center contains $t$ is symmetric.
Therefore there are at most $r$ Berge triangles containing $e$ with $s,t$ as their core vertices on $e$. Summing over the $\binom{r+1}{2}$ possible pairs gives at most $r\binom{r+1}{2}=T_r$ Berge triangles containing $e$. Since distinct coupled vertices give rise to distinct such triangles, the second statement follows.
\end{proof}

\section{Proof of \Cref{thm:maintriangle}}\label{sec:triangle}

We first justify the assertion about good orders in triangle-free hypergraphs.
Suppose that a total order of a triangle-free hypergraph violates \eqref{eq:good}.
There are an edge $f$, a vertex $x\notin f$, and distinct $u,z\in f$ with $E(x,u),E(x,z)\ne\varnothing$.
Relabel so that $x\prec u\prec z$, and choose $e_u\in E(x,u)$ and $e_z\in E(x,z)$.
We have $z\notin e_u$ because its head is $u$, and $u\notin e_z$ because its neck is $x$.
Together with $x\notin f$, these exclusions show that the three distinct edges $f,e_u,e_z$ form a triangle, a contradiction.

Now fix a hypergraph $G=(V,E)$ with a good order $\prec$ on $V$ and  $a\in (0,1)^V$.
Set
\[
p\colon [0,\infty)\to [0,1),\quad p(t)=1-e^{-t}.
\]
Apply the ordered random pick process to $G$, $\prec$, $a$ and $p$ and obtain a random independent set $I_n$.
We shall prove that for every $v \in V$,
\[
\P(v\in I_n)\geq (1-e^{-a(v)})\left( 1+\dfrac{\Lambda_a(v)}{\log a(v)} \right).
\]

\if0
To start with, fix $v=v_k$ and consider the following \emphd{modified process} to $G$, $\prec$, $a$, $p$ and $v$:
initialize $\tilde a_0(v)=a(v)$ and $\tilde I_0=\0$, and for step $i=1,2,\ldots,k-1$, put
\[
 \tilde L_i(v_j):=\mathbf{1}_{\{\exists e\in E(v_i,v_j):\ e\setminus\{v_i,v_j\}\subseteq\tilde I_{i-1}\}}
 \qquad(j>i).
\]
\begin{enumerate}
    \item If $\tilde L_i(v)=0$, do the same update rule as in the ordered random pick process: \todo{tbc}
    \item If $\tilde L_i(v)=1$, put $\tilde I_i=\tilde I_{i-1}$. For $j > i$, set $\tilde a_i(v_j)=e^{\tilde a_{i-1}(v_i)}a_{i-1}(v_j)$ for all $v_j$ with $L_i(v_j)=1$, and otherwise $\tilde a_i(v_j)=\tilde a_{i-1}(v_j)$.
\end{enumerate}
\fi
Fix a target $v=v_k$ and run \Cref{alg:modified-process} with $p(t)=1-e^{-t}$.

\begin{algorithm}[!htbp]
\caption{Modified process with target $v=v_k$}
\label{alg:modified-process}
\begingroup
\raggedright

\setlength{\abovedisplayskip}{5pt}
\setlength{\belowdisplayskip}{5pt}
\setlength{\abovedisplayshortskip}{3pt}
\setlength{\belowdisplayshortskip}{4pt}

\noindent\textbf{Input:}
A hypergraph $G=(V,E)$ with an order
$v_1\prec\cdots\prec v_n$, an initial activity vector
$a\in(0,1)^V$, a measurable selection function
$p:[0,\infty)\to[0,1)$ with $p(0)=0$, and a target vertex $v=v_k$.

\noindent\textbf{Output:}
The history $(\widetilde I_i,\widetilde a_i)_{0\le i<k}$
before the target is processed.

\begin{enumerate}[itemsep=.15cm,topsep=5pt]
\item
Set
\[
\widetilde I_0=\varnothing,
\qquad
\widetilde a_0(u)=a(u)\quad (u\in V).
\]

\item
For $i=1,\ldots,k-1$, put
\[
\widetilde q_i:=p(\widetilde a_{i-1}(v_i)),
\]
and, for every $j>i$, define
\[
\widetilde L_i(v_j):=
\mathbf 1_{\{
\exists e\in E(v_i,v_j):
e\setminus\{v_i,v_j\}\subseteq \widetilde I_{i-1}
\}}.
\]
For every $j>i$ with $\widetilde L_i(v_j)=0$, set
\[
\widetilde a_i(v_j)=\widetilde a_{i-1}(v_j).
\]
If $\widetilde L_i(v)=0$, select $v_i$ with probability
$\widetilde q_i$, conditionally on the previous choices.
If $\widetilde L_i(v)=1$, force $v_i$ to be unselected.
If $v_i$ is selected, set
\[
\widetilde I_i=\widetilde I_{i-1}\cup\{v_i\},
\]
and, for every $j>i$ with $\widetilde L_i(v_j)=1$, set
\[
\widetilde a_i(v_j)=0.
\]
If $v_i$ is not selected, set
\[
\widetilde I_i=\widetilde I_{i-1},
\]
and, for every $j>i$ with $\widetilde L_i(v_j)=1$, set
\[
\widetilde a_i(v_j)
=
\frac{\widetilde a_{i-1}(v_j)}
{1-\widetilde q_i}.
\]

\item
Return the history through step $k-1$.
\end{enumerate}

\endgroup
\end{algorithm}

The modified process agrees with the original process except when $\widetilde L_i(v)=1$. In that case, it deterministically leaves $v_i$ unselected and applies the same mass updates as the unselected branch of the original process.
Define
    \[
    X_v \defeq \sum_{i<k}\widetilde a_{i-1}(v_i)\widetilde L_i(v),
    \]
Directly from the modified process, we have
\[
\widetilde a_{k-1}(v)=a(v)\exp\left(\sum_{i<k,\, \widetilde L_i(v)=1} \widetilde a_{i-1}(v_i) \right)=a(v)\exp\left(  \sum_{i<k} \widetilde a_{i-1}(v_i) \widetilde L_i(v) \right)=a(v)e^{X_v}.
\]

We may use the modified process to give a lower bound on $\P(v\in I_n)$.
Let $\tE$ denote the expectation for the modified process, and let $\widetilde \F_i$ be the $\sigma$-algebra generated by its first $i$ choices. 
\begin{claim}\label{clm:PvinIn}
We have
    \[
    \P(v\in I_n)\geq (1-e^{-a(v)})\left(1+\dfrac{\tE[X_v]}{\log a(v)}\right).
    \]
\end{claim}

\begin{claimproof}
Following the fixed-target comparison in~\cite{Martinsson2025RandomInd}, we compare the two processes up to the time when $v=v_k$ is processed.

Pair each history of the modified process through step $k-1$ with the history of the original process making the same branch choices.
Inductively, the two processes have the same selected set and current masses along each pair of corresponding histories.
If $\widetilde L_i(v)=0$, their transition probabilities and branch updates agree. 
If $\widetilde L_i(v)=1$, the modified process leaves $v_i$ unselected with probability one. The corresponding unselected branch in the original process has conditional probability $e^{-\widetilde a_{i-1}(v_i)}$ and gives the same mass updates; selecting $v_i$ instead would set the target's mass to zero.
Thus, for each history of the modified process through step $k-1$, the corresponding original history has probability equal to its modified probability multiplied by
\[
\prod_{i<k}e^{-\widetilde a_{i-1}(v_i)\widetilde L_i(v)}=e^{-X_v}.
\]
Summing over the histories of the modified process and using the selection probability at step $k$, we obtain
\begin{align*}
    \P(v\in I_n)
    =\tE\!\left[e^{-X_v}p(\widetilde a_{k-1}(v)) \right]
    =\tE\!\left[ e^{-X_v}\bigl(1-e^{-a(v)e^{X_v}}\bigr)
    \right].
\end{align*}
Applying \Cref{ineq:e-t} pointwise and then taking expectation gives
\[
\P(v\in I_n)\geq(1-e^{-a(v)})\left(1+\frac{\tilde E X_v}{\log a(v)} \right).\qedhere
\]    
\end{claimproof}

To complete the proof, it remains to show that
\[
\tE[X_v]=\tE\left[\sum_{i<k}\widetilde a_{i-1}(v_i)\widetilde L_i(v)\right]\leq \Lambda_a(v)=\sum_{e\in \B(v)}
    a(m^*(e)) \prod_{u\in e\setminus\set{m^*(e),v}}
    \left(1-e^{-a(u)}\right).
\]
This is where the hypergraph setting requires a new ingredient. In the graph case, the corresponding quantity is controlled by the mass of a single earlier neighbor. For a hyperedge, however, its earlier vertices must all have been selected before the neck is processed. We therefore estimate the contribution of each backward edge separately.

\begin{claim}\label{clm:zie}
    Fix $1\leq i<k$ and $e\in E(v_i,v)$.
    Put $S=e\setminus\{v_i,v\}$ and
    \[
    Z_i(e) \defeq \widetilde a_{i-1}(v_i)\indicator_{\set{S \subseteq \widetilde I_{i-1}}}.
    \]
    Then 
    \[
    \tE[Z_i(e)]\leq a(v_i)\prod_{u\in S}\left(1-e^{-a(u)}\right).
    \]
\end{claim}

\begin{claimproof}
 For $0\leq j<i$, define
\[
    Y_j\defeq\widetilde a_j(v_i)
    \indicator_{\{S\cap V_j\subseteq\widetilde I_j\}}
    \prod_{u\in S\setminus V_j}\left(1-e^{-\widetilde a_j(u)}\right).
\]
In $Y_j$, processed vertices contribute their actual selection indicators, while each unprocessed vertex $u\in S\setminus V_j$ contributes the current value
\[
p(\widetilde a_j(u))=1-e^{-\widetilde a_j(u)}
\]
of the selection function. Thus $Y_j$ interpolates between the deterministic initial product $Y_0$ and $Z_i(e)=Y_{i-1}$.
We will show that $(Y_j)_{0\le j<i}$ is a supermartingale.

Since $S\subseteq V_{i-1}$, we have
\[
    Y_0=a(v_i)\prod_{u\in S}(1-e^{-a(u)}),
    \qquad
    Y_{i-1}=Z_i(e).
\]

We prove that
\[
\tE
[Y_j\mid\widetilde{\mathcal F}_{j-1}]
\leq Y_{j-1}
\qquad (1\leq j<i).
\]
Condition on $\widetilde{\mathcal F}_{j-1}$, we distinguish two cases. Set $s:=\widetilde a_{j-1}(v_j)$.

\begin{enumerate}

\item Suppose that $v_j\notin S$.
Since $v_j\notin e$ and the order is good, at most one of
\[
\widetilde L_j(v),\quad
\widetilde L_j(v_i),\quad
\widetilde L_j(u),\; u\in S\setminus V_j
\]
is equal to $1$. Moreover, $S\cap V_j=S\cap V_{j-1}$.

\begin{itemize}

\item If $\widetilde L_j(v)=1$, the modified process forces $v_j$ to be unselected.
By goodness, none of the factors in $Y_j$ changes, so $Y_j=Y_{j-1}$.
The same holds if all the above indicators are zero.

\item If $\widetilde L_j(v_i)=1$, the step
is not forced, and only the mass of $v_i$
may change. Thus $Y_j=0$ on the selected
branch and $Y_j=e^sY_{j-1}$ on the
unselected branch. Consequently,
\[
\begin{aligned}
\tE[Y_j\mid\widetilde{\mathcal F}_{j-1}]
=(1-e^{-s})\cdot0+e^{-s}e^sY_{j-1}=Y_{j-1}.
\end{aligned}
\]

\item If $\widetilde L_j(u)=1$ for some $u\in S\setminus V_j$, the step is not forced, and only the factor $1-e^{-\widetilde a_j(u)}$ may change.
By concavity of $p(t)=1-e^{-t}$,
\[
\begin{aligned}
\tE[1-e^{-\widetilde a_j(u)}
\mid\widetilde{\mathcal F}_{j-1}]\leq
1-\exp\left(-\tE[\widetilde a_j(u)\mid\widetilde{\mathcal F}_{j-1}]\right)
=1-e^{-\widetilde a_{j-1}(u)}.
\end{aligned}
\]
Since all other factors are unchanged,
\[
\tE
[Y_j\mid\widetilde{\mathcal F}_{j-1}]
\leq Y_{j-1}.
\]

\end{itemize}

\item Suppose that $v_j\in S$.

\begin{itemize}

\item If $\widetilde L_j(v)=1$, the
modified process forces $v_j$ to be
unselected, so $Y_j=0$.

\item If $\widetilde L_j(v)=0$, write
\[
Y_{j-1}=(1-e^{-s})R,
\]
where
\[
R:=
\widetilde a_{j-1}(v_i)
\indicator_{\{
S\cap V_{j-1}\subseteq\widetilde I_{j-1}
\}}
\prod_{u\in S\setminus V_j}
(1-e^{-\widetilde a_{j-1}(u)}).
\]
On the unselected branch, $Y_j=0$.
On the selected branch,
\[
\widetilde a_j(v_i)
\leq\widetilde a_{j-1}(v_i),
\qquad
\widetilde a_j(u)
\leq\widetilde a_{j-1}(u)
\quad (u\in S\setminus V_j).
\]
Since $p$ is nondecreasing, $Y_j\leq R$.
Therefore,
\[
\begin{aligned}
\tE
[Y_j\mid\widetilde{\mathcal F}_{j-1}]
\leq (1-e^{-s})R+e^{-s}\cdot0=Y_{j-1}.
\end{aligned}
\]

\end{itemize}
\end{enumerate}

In either case,
\[
\tE
[Y_j\mid\widetilde{\mathcal F}_{j-1}]
\leq Y_{j-1}.
\]
Iterating gives
\[
\tE Z_i(e)=\tE Y_{i-1}
\leq Y_0=a(v_i)
\prod_{u\in e\setminus\{v_i,v\}}
(1-e^{-a(u)}). \hfill \qedhere
\]
\end{claimproof}

Finally, by the definition of
$\widetilde L_i(v)$,
\[
\widetilde a_{i-1}(v_i)\widetilde L_i(v)
\leq\sum_{e\in E(v_i,v)}Z_i(e).
\]
Summing over $i<k$ and applying
\Cref{clm:zie}, we obtain
\[
\begin{aligned}
\tE X_v
&\leq
\sum_{i<k}\sum_{e\in E(v_i,v)}
\tE Z_i(e)\\
&\leq
\sum_{e\in\mathcal B(v)}
a(m^*(e))
\prod_{u\in e\setminus\{m^*(e),v\}}
(1-e^{-a(u)})\\
&=\Lambda_a(v).
\end{aligned}
\]
Since $\log a(v)<0$, \Cref{clm:PvinIn} yields
\[
\mathbb P(v\in I_n)
\geq
(1-e^{-a(v)})
\left(
1+\frac{\Lambda_a(v)}{\log a(v)}
\right).
\]
Since the original process is common
to all targets, the bound holds
simultaneously for every vertex.
This completes the proof of
\Cref{thm:maintriangle}.

\section{Proof of \Cref{thm:mainC4}}\label{sec:C4}
Let $G=(V,E)$ be a finite linear Berge-$C_4$-free $(r+1)$-uniform hypergraph with arbitrary total order $v_1\prec\cdots\prec v_n$ and initial activities $a\in (0,1)^V$.
For $\rho\in (0,1)$, set 
\[
p_\rho\colon [0,\infty)\to [0,1),\quad p_\rho(t)\defeq\min\{t,\rho\}.
\]

Recall
\[
    T_r=r\binom{r+1}{2},
    \qquad
    \kappa_\rho=(1-\rho)^{-1-rT_r}.
\]
Fix $k\in [n]$ and set $v=v_k$.

\begin{lemma}\label{lem:c4-moments}
Apply the ordered random pick process to $G$, $\prec$, $a$ and $p_\rho$, and obtain a random independent set $I_n$. Then
\[
    \E a_{k-1}(v)=a(v),
    \qquad
    \P(v\in I_n)=\E p_\rho(a_{k-1}(v)),
    \qquad
    \E a_{k-1}(v)^2\leq a(v)^2\exp\big(\kappa_\rho\Gamma_a(v)\big).
\]
\end{lemma}
\begin{proof}
    The first two identities follow from \Cref{prop:process}. We prove the second-moment estimate in three steps.
    \begin{claim}\label{clm:square-step}
    For $1\leq i\leq k-1$,
        \begin{equation*}
        \E[a_i(v)^2\mid\mathcal F_{i-1}]
  \leq a_{i-1}(v)^2
  \left(1+\frac{a_{i-1}(v_i)}{1-\rho}L_i(v)\right).
    \end{equation*}
    %Here if $ E(v_i,v)=\0$, the product is set to be 0, and note that since $G$ is linear, $|E(v_i,v)|\leq 1$.
    \end{claim}
    \begin{claimproof}
Condition on $\F_{i-1}$ and put $q=p_\rho(a_{i-1}(v_i))$ and $z=a_{i-1}(v)$. If $L_i(v)=0$, the target mass is unchanged. If $L_i(v)=1$, it becomes zero with probability $q$ and $z/(1-q)$ with probability $1-q$. Thus
\[
 \E[a_i(v)^2\mid\F_{i-1}]
 =z^2\left(1+\frac q{1-q}L_i(v)\right)
 \leq z^2\left(1+\frac{a_{i-1}(v_i)}{1-\rho}L_i(v)\right),
\]
since $q\leq a_{i-1}(v_i)$ and $1-q\geq1-\rho$.
    \end{claimproof}
   % Upon this iteration, we use the following notion to record the expectation of $a_{k-1}(v)^2$ represented by the information in $\F_i$.
    For $e\in\B(v)$ and $0\leq i<k$, abbreviate
\[
 Z_i(e):=\mathbf 1_{\{e^-\cap V_i\subseteq I_i\}}
 \prod_{u\in e^-\setminus V_i}a_i(u).
\]
For $S\subseteq\B(v)$, define
\[
 Y_i(S):=a_i(v)^2\prod_{e\in S}Z_i(e).
\]
    In particular,
    \[
    Y_{k-1}(\0)=a_{k-1}(v)^2,\qquad Y_0(S)=a(v)^2\prod_{e\in S}\prod_{u\in e^-}a(u).
    \]
    For $1\leq i<k$ and $S\subseteq\B(v)$, define
\[
C_i(S):=
\mathbf 1_{\{\exists e\in S:\ (v_i,e)\text{ is coupled}\}},
\qquad
\beta_i(S):=(1-\rho)^{-rC_i(S)}.
\]
    \begin{claim}\label{claim:YiS}
    For $1\leq i\leq k-1$ and $S\subseteq \B(v)$, we have
        \begin{align*}
       \E[Y_i(S)\mid \F_{i-1}]\leq 
       \begin{cases}
           \beta_i(S)\big(Y_{i-1}(S)+(1-\rho)^{-1}Y_{i-1}(S\cup \set{f})\big), &  E(v_i,v)\setminus S=\{f\};\\
           \beta_i(S)\big(Y_{i-1}(S)\big),& \text{otherwise}.
        \end{cases}
    \end{align*}
    \end{claim}

    \begin{claimproof}
Condition on $\F_{i-1}$ and write
\[
s:=a_{i-1}(v_i),\qquad
q:=p_\rho(s).
\]

First suppose that $v_i$ belongs to no edge of $S$. Put
\[
P_{i-1}:=\prod_{e\in S}Z_{i-1}(e),
\qquad
P_i:=\prod_{e\in S}Z_i(e).
\]
By \Cref{lem:c4-separation}\ref{lem:c4-separation1},
all affected tail masses appearing in these products belong to at most one edge of $S$.
Let $m$ be their number, so $0\leq m\leq r$.

Suppose first that $E(v_i,v)=\varnothing$. Then the target mass is unchanged.
If $m=0$, then $Y_i(S)=Y_{i-1}(S)$. If $m\geq1$, the selected branch gives $P_i=0$, while on the unselected branch
\[
P_i=(1-q)^{-m}P_{i-1}.
\]
Hence
\[
\E[Y_i(S)\mid\F_{i-1}]
=(1-q)^{1-m}Y_{i-1}(S).
\]
If $m=1$, this equals $Y_{i-1}(S)$. If $m\geq2$, then the affected edge is coupled with $v_i$, so $C_i(S)=1$. Since $q\leq\rho$ and
$m\leq r$,
\[
(1-q)^{1-m}\leq(1-\rho)^{-r}=\beta_i(S).
\]
Thus
\[
\E[Y_i(S)\mid\F_{i-1}]
\leq\beta_i(S)Y_{i-1}(S).
\]

Now suppose that $E(v_i,v)=\{f\}$. Since $v_i$ belongs to no edge of $S$, we have $f\notin S$. If $m\geq1$, then the affected edge in $S$ is coupled with $v_i$, since $x\in f^-$. Consequently, whether $m=0$ or $m\geq1$, both branches satisfy
\[
P_i\leq\beta_i(S)P_{i-1}.
\]
Using \Cref{clm:square-step},
\begin{align*}
\E[Y_i(S)\mid\F_{i-1}]
&\leq
\beta_i(S)P_{i-1}
\E[a_i(v)^2\mid\F_{i-1}]\\
&\leq
\beta_i(S)
\left(
Y_{i-1}(S)
+\frac{sL_i(v)}{1-\rho}Y_{i-1}(S)
\right).
\end{align*}
Since $v_i=m^*(f)$,
\[
Z_{i-1}(f)=sL_i(v).
\]
As $f\notin S$,
\[
sL_i(v)Y_{i-1}(S)
=
Y_{i-1}(S\cup\{f\}).
\]
Therefore
\[
\E[Y_i(S)\mid\F_{i-1}]
\leq
\beta_i(S)
\bigl(
Y_{i-1}(S)
+(1-\rho)^{-1}Y_{i-1}(S\cup\{f\})
\bigr).
\]

It remains to consider the case that $v_i\in f^-$ for some $f\in S$.
By linearity, such an edge $f$ is unique. If
\[
f^-\cap V_{i-1}\not\subseteq I_{i-1},
\]
then both $Y_{i-1}(S)$ and $Y_i(S)$ are zero, so we may assume
\[
f^-\cap V_{i-1}\subseteq I_{i-1}.
\]

If $v_i=m^*(f)$, then $Y_i(S)=0$ on both branches. Indeed, if $v_i$ is not selected, then
\[
f^-\cap V_i\not\subseteq I_i,
\]
so $Z_i(f)=0$; if $v_i$ is selected, then $L_i(v)=1$ and hence
$a_i(v)=0$.

Suppose therefore that $v_i\neq m^*(f)$. By linearity,
\[
E(v_i,v)=\varnothing,
\]
since any edge in $E(v_i,v)$ would share the two vertices $v_i$ and $v$ with $f$. Similarly, no other tail mass of $f$ can be affected
by $v_i$.

Write
\[
Y_{i-1}(S)=sR,
\]
where $R$ is the product defining $Y_{i-1}(S)$ with the single factor $s=a_{i-1}(v_i)$ removed. On the unselected branch, $Z_i(f)=0$, so $Y_i(S)=0$. On the selected branch, the factor $s$ is removed, the other masses in $f^-$ are unchanged, and masses belonging to other edges of $S$ can only stay unchanged or become zero. Hence
\[
Y_i(S)\leq R
\]
on the selected branch. Since $q=p_\rho(s)\leq s$,
\[
\E[Y_i(S)\mid\F_{i-1}]
\leq qR
\leq sR
=
Y_{i-1}(S).
\]
Here $E(v_i,v)\setminus S=\varnothing$ and $\beta_i(S)\geq1$, so the second case of the claim follows.
    \end{claimproof}

Using \Cref{claim:YiS}, we can obtain an upper bound on $\E a_{k-1}(v)^2$.
\begin{claim}\label{cl:coefficients}
We have
\[
 \E a_{k-1}(v)^2\leq
 \sum_{S\subseteq\B(v)}\kappa_\rho^{|S|}Y_0(S).
\]
\end{claim}
\begin{claimproof}
For $e\in\B(v)$, let $\nu(e)$ be the index of its neck, so that $m^*(e)=v_{\nu(e)}$. The indices $\nu(e)$ are distinct: two backward edges with the same neck would share both that neck and $v$. For $0\leq i<k$, put
\[
 \B_i:=\{e\in\B(v):\nu(e)>i\}.
\]
Thus $\B_0=\B(v)$ and $\B_{k-1}=\varnothing$. If an edge $f_i\in \B(v)$ has neck $v_i$, then $\B_{i-1}=\B_i\cup\{f_i\}$; otherwise $\B_{i-1}=\B_i$.

We define nonnegative deterministic coefficients $c_{S,i}$, supported on $S\subseteq\B_i$, so that
\begin{equation}\label{eq:conditional-iteration}
 \E[a_{k-1}(v)^2\mid\F_i]
 \leq\sum_{S\subseteq\B_i}c_{S,i}Y_i(S).
\end{equation}
At $i=k-1$, take $c_{\varnothing,k-1}=1$.
Suppose the coefficients at time $i$ are defined. For every $S\subseteq\B_i$, set
\begin{equation}\label{eq:coefficient-recursion}
 \begin{aligned}
 c_{S,i-1}&:=\beta_i(S)c_{S,i},\\
 c_{S\cup\{f_i\},i-1}&:=(1-\rho)^{-1}\beta_i(S)c_{S,i}
       &&\text{if }f_i\text{ exists}.
 \end{aligned}
\end{equation}
The second line is omitted when $f_i$ does not exist. There is no conflict between these assignments: $f_i\notin\B_i$, so the first line defines coefficients of sets not containing $f_i$ and the second defines those containing it. The tower property and \Cref{claim:YiS} prove \eqref{eq:conditional-iteration} at time $i-1$.

Fix a final subset $S\subseteq\B(v)$. In this backward expansion there is exactly one sequence of subset choices leading to $S$: an edge $e$ is inserted at step $\nu(e)$ if and only if $e\in S$. Thus the subset at time $i$ on that sequence is $S\cap\B_i$. Every insertion contributes $(1-\rho)^{-1}$, and every step contributes $\beta_i(S\cap\B_i)$. Therefore
\[
 c_{S,0}=(1-\rho)^{-|S|}
 \prod_{i=1}^{k-1}\beta_i(S\cap\B_i)
 =(1-\rho)^{-|S|-r\sum_{i=1}^{k-1}C_i(S\cap\B_i)}.
\]
For this sequence,
\begin{align*}
 \sum_{i=1}^{k-1}C_i(S\cap\B_i)
 &\leq\sum_{e\in S}
 \bigl|\{i<k:(v_i,e)\text{ is coupled}\}\bigr|\\
 &\leq T_r|S|
\end{align*}
by \Cref{lem:c4-separation}(2).
Notice that the bound is $T_r$ per edge, not $T_r$ for the whole subset.
It follows that
\[
 c_{S,0}\leq(1-\rho)^{-(1+rT_r)|S|}=\kappa_\rho^{|S|}.
\]
Taking $i=0$ in \eqref{eq:conditional-iteration}, where $\F_0$ is trivial, proves the claim.
\end{claimproof}

Finally, \Cref{cl:coefficients} and the expansion of a product over subsets give
\begin{align*}
 \E a_{k-1}(v)^2
 &\leq\sum_{S\subseteq\B(v)}\kappa_\rho^{|S|}Y_0(S)\\
 &=a(v)^2\sum_{S\subseteq\B(v)}
   \prod_{e\in S}\left(\kappa_\rho\prod_{u\in e^-}a(u)\right)\\
 &=a(v)^2\prod_{e\in\B(v)}
   \left(1+\kappa_\rho\prod_{u\in e^-}a(u)\right)\\
 &\leq a(v)^2\exp\!\left(\kappa_\rho\sum_{e\in\B(v)}\prod_{u\in e^-}a(u)\right)
 =a(v)^2e^{\kappa_\rho\Gamma_a(v)}.\qedhere
\end{align*}

\end{proof}
We now return to \Cref{thm:mainC4}, keeping the total order fixed. For $\varepsilon\in(0,1)$, choose
\[
 \rho:=\frac{\varepsilon}{2(T_r+1)},\qquad
 \delta:=\frac{\varepsilon\rho}{2}
 =\frac{\varepsilon^2}{4(T_r+1)},
\]
and define new initial activities
\[
 a'(u):=(1-\varepsilon/2)\min\{a(u),\delta e^{-\Gamma_a(u)}\}.
\]
Run the ordered process afresh with $a'$ and $p_\rho$, write $a'_i$ for its evolving masses, and let $I'_n$ be its output. All probabilities and expectations in the rest of this proof refer to this new run. Since $\min\{s,\rho\}\geq s-s^2/\rho$ for $s\geq0$, \Cref{lem:c4-moments} gives
\begin{align*}
 \P(v\in I'_n)
 &=\E p_\rho(a'_{k-1}(v))\\
 &\geq a'(v)-\frac{\E[a'_{k-1}(v)^2]}\rho\\
 &\geq a'(v)\left(1-\frac{a'(v)e^{\kappa_\rho\Gamma_{a'}(v)}}\rho\right).
\end{align*}
Each tail has exactly $r$ vertices, so $a'(u)\leq(1-\varepsilon/2)a(u)$ implies
\[
 \Gamma_{a'}(v)\leq(1-\varepsilon/2)^r\Gamma_a(v).
\]
Moreover, $1-\varepsilon/2=1-(T_r+1)\rho\leq(1-\rho)^{T_r+1}$ by Bernoulli's inequality. Consequently,
\begin{align*}
 \kappa_\rho\Gamma_{a'}(v)
 &\leq(1-\rho)^{-1-rT_r+r(T_r+1)}\Gamma_a(v)\\
 &=(1-\rho)^{r-1}\Gamma_a(v)\leq\Gamma_a(v).
\end{align*}
It follows that
\[
 \frac{a'(v)e^{\kappa_\rho\Gamma_{a'}(v)}}\rho
 \leq\frac\delta\rho
 e^{-\Gamma_a(v)+\kappa_\rho\Gamma_{a'}(v)}
 \leq\frac\varepsilon2.
\]
Hence
\[
 \P(v\in I'_n)
 \geq(1-\varepsilon/2)^2\min\{a(v),\delta e^{-\Gamma_a(v)}\}
 \geq(1-\varepsilon)\min\{a(v),\delta e^{-\Gamma_a(v)}\}.
\]
The same new run gives these bounds for all vertices, completing the proof of \Cref{thm:mainC4}.

\section{Proofs of corollaries}\label{sec:corollaries}

\begin{proof}[Proof of \Cref{cor:localshearerbound}]
Fix the auxiliary weight $w$, put $d_v=d_w(v)$ and $a(v)=\ell_r(d_v)$, and order $V$ by nondecreasing values of $a(v)/w(v)^{1/r}$. If $u\prec v$, then
\[
 a(u)\leq a(v)\left(\frac{w(u)}{w(v)}\right)^{1/r}.
\]
Since $G$ is triangle-free, this order is good.
For every $e\in\B(v)$,
\begin{align*}
 a(m^*(e))\prod_{u\in e\setminus\{m^*(e),v\}}(1-e^{-a(u)})
 &\leq\prod_{u\in e^-}a(u)\\
 &\leq a(v)^r
 \left(\prod_{u\in e\setminus\{v\}}\frac{w(u)}{w(v)}\right)^{1/r}.
\end{align*}
Summing and enlarging from $\B(v)$ to all edges containing $v$ gives
$\Lambda_a(v)\leq a(v)^r d_v$. 
Since $\log a(v)<0$, 
\Cref{thm:maintriangle} yields
\[
 \P(v\in I)\geq(1-e^{-a(v)})
 \left(1+\frac{a(v)^r d_v}{\log a(v)}\right)
 =F_r(d_v).
\]
Positivity and the asymptotic formula follow from \Cref{lem:Frx}.
\end{proof}

\begin{proof}[Proof of \Cref{cor:fractriangle}]
    
Choose an order with $|\B(v)|\leq d$. Put $\alpha=\ell_r(d)$ and define the activity vector by $a(v)=\alpha$ for all $v\in V$. This order is good and
\[
 \Lambda_a(v)\leq d\alpha(1-e^{-\alpha})^{r-1}\leq d\alpha^r.
\]
\Cref{thm:maintriangle} therefore gives a common marginal lower bound $F_r(d)$. By \Cref{lem:Frx}, 
\[
 \chi_f(G)\leq F_r(d)^{-1}
 =\left(\frac{r+1}{r}(r(r+1))^{1/r}+o_r(1)\right)
 \left(\frac d{\log d}\right)^{1/r}.\qedhere
\]
\end{proof}

\begin{proof}[Proof of \Cref{thm:c4-weighted-local}]
Fix $\varepsilon\in(0,1)$, set $\theta=\varepsilon/2$, and let $\delta_0=\delta(r,\theta)$ be supplied by \Cref{thm:mainC4}. For $x>1$, put
\[
 \beta(x):=\left(\frac{(1-\theta)\log x}{rx}\right)^{1/r}.
\]
Choose $D=D(r,\varepsilon)>1$ so large that, for all $x\geq D$,
\begin{equation}\label{eq:local-cutoff}
 \beta(x)<1,\qquad
 \beta(x)x^{(1-\theta)/r}
 =\left(\frac{1-\theta}{r}\log x\right)^{1/r}x^{-\theta/r}
 \leq\delta_0.
\end{equation}
Such a choice is possible because the last expression tends to zero.

Write $d_v=d_w(v)$, computed in the original hypergraph, and define an activity vector on all of $V$ by
\[
 a(v):=
 \begin{cases}
 \beta(d_v),&d_v\geq D,\\
 1/2,&d_v<D.
 \end{cases}
\]
Order the vertices by nondecreasing $a(v)/w(v)^{1/r}$. For $u\prec v$, this order gives
\[
 a(u)\leq a(v)\left(\frac{w(u)}{w(v)}\right)^{1/r}.
\]
Multiplying over a backward edge and summing therefore yields
\[
 \Gamma_a(v)
 \leq a(v)^r\sum_{e\in\B(v)}
 \left(\prod_{u\in e^-}\frac{w(u)}{w(v)}\right)^{1/r}
 \leq a(v)^r d_v.
\]
If $d_v\geq D$, this implies
\[
 a(v)e^{\Gamma_a(v)}
 \leq \beta(d_v)d_v^{(1-\theta)/r}\leq\delta_0.
\]
Thus the minimum in \Cref{thm:mainC4}, applied with error parameter $\theta$, equals $a(v)$ at every such vertex. One random independent set therefore satisfies
\begin{align*}
 \P(v\in I)
 &\geq(1-\theta)a(v)\\
 &=(1-\theta)^{1+1/r}
 \left(\frac{\log d_v}{rd_v}\right)^{1/r}
 \geq(1-\varepsilon)
 \left(\frac{\log d_v}{rd_v}\right)^{1/r}.
\end{align*}
The last inequality uses $1+1/r\leq2$ and $(1-\theta)^2\geq1-2\theta$. The value $1/2$ below the cutoff ensures that $a \in (0, 1)^V$; no marginal requirement is imposed at those vertices. All local degrees above are those of the original hypergraph.
\end{proof}

\begin{proof}[Proof of \Cref{cor:c4-fractional}]
Fix $\theta\in(0,1)$ and let $\delta_0=\delta(r,\theta)$. If $V=\varnothing$, the assertion is immediate. Otherwise choose a degeneracy order and set $a(v)=\alpha$ for every vertex, where
\[
 \alpha:=\left(\frac{(1-\theta)\log d}{rd}\right)^{1/r}.
\]
For sufficiently large $d$, $\alpha\in(0,1)$ and
\[
 \alpha\exp(d\alpha^r)
 =\left(\frac{1-\theta}{r}\log d\right)^{1/r}d^{-\theta/r}
 \leq\delta_0.
\]
Since $\Gamma_a(v)\leq d\alpha^r$ for every vertex, \Cref{thm:mainC4} gives the common marginal bound $(1-\theta)\alpha$. Consequently,
\[
 \chi_f(G)\leq\frac1{(1-\theta)\alpha}
 =(1-\theta)^{-1-1/r}
 \left(\frac{rd}{\log d}\right)^{1/r}.
\]
For every fixed $\theta>0$ this holds once $d$ is sufficiently large. Letting $\theta$ tend to zero gives \begin{equation*}
    \chi_f(G)
    \leq
    \bigl(1+o_r(1)\bigr)
    \left(\frac{rd}{\log d}\right)^{1/r}
    \qquad(d\to\infty).\qedhere
\end{equation*}
\end{proof}

\section*{AI declaration}

The authors used ChatGPT with GPT-5.6 Sol and GPT-6 Astra in the development of \Cref{thm:mainC4}. Specifically, the AI tools assisted with the choice of the selection function $p_\rho$ and the correct iteration family $\{Y_i(S)\}$ under the recurrence relation in \Cref{clm:square-step}. Other mathematical arguments are due to the authors. ChatGPT was also used to assist with the language, organization, and presentation of the manuscript. The authors take full responsibility for the contents.

\printbibliography

\end{document}